\documentclass[12pt,a4paper,reqno]{amsart}
\usepackage{amsmath}
\usepackage{amsthm}
\usepackage{amsfonts}
\usepackage{amssymb}
\usepackage{bm}
\numberwithin{equation}{section}

\theoremstyle{plain}

\newtheorem{theorem}[subsection]{Theorem}

\newtheorem{lemma}[subsection]{Lemma}

\theoremstyle{definition}

\renewcommand{\leq}{\leqslant}
\renewcommand{\geq}{\geqslant}

\newsavebox{\proofbox}
\savebox{\proofbox}{\begin{picture}(7,7)%
  \put(0,0){\framebox(7,7){}}\end{picture}}

\def\E{\mathbb{E}}

\def\T{\mathbb{T}}
\def\C{\mathbb{C}}

\def\F{\mathbb{F}}

\begin{document}

\title{A new proof of Roth's theorem on arithmetic progressions
}

\author{Ernie Croot}
\address{Department of Mathematics\\
     Georgia Institute of Technology \\
     103 Skiles\\
     Atlanta, Georgia 30332 \\
     U.S.A.}
\email{ecroot@math.gatech.edu}
\author{Olof Sisask}
\address{Department of Mathematics\\
     University of Bristol\\
     Bristol BS8 1TW\\
     England}
\email{O.Sisask@dpmms.cam.ac.uk}

\thanks{The first author is funded by NSF grant 
DMS-0500863. The second author is funded by an 
EPSRC DTG through the University of Bristol, and 
would like to thank the University of Cambridge for
its kind hospitality while this work was carried out.}

\begin{abstract}
We present a proof of Roth's theorem that follows a slightly different 
structure to the usual proofs, in that there is not much iteration.
Although our proof works using a type of density 
increment argument (which is typical of most proofs of Roth's theorem), 
we do not pass to a progression related to the large Fourier coefficients
of our set (as most other proofs of Roth do).  Furthermore, in our proof, 
the density increment is achieved through an application of a quantitative 
version of Varnavides's theorem, which is perhaps unexpected. 
\end{abstract}
\maketitle
     \section{Introduction}
Given an integer $N \geq 1$, let $r_3(N)$ denote the size of any largest
subset $S$ of $[N] := \{1,...,N\}$ for which there are no solutions to
$$
x + y = 2z,\ x,y,z \in S,\ x \neq y;
$$
in other words, $S$ has no non-trivial three term arithmetic progressions.

In the present paper we give a proof of Roth's theorem \cite{roth}
that, although iterative, uses 
a more benign type of iteration than most proofs. 

\begin{theorem}\label{main_theorem}  We have that $r_3(N) = o(N)$.
\end{theorem}

Roughly, we achieve this by showing that $r_3(N)/N$ is 
asymptotically decreasing.  We will do this by starting with a set 
$S \subseteq [N]$, $|S| = r_3(N)$, such that $S$ has no three term
progressions, and then convolving it with a measure on a carefully
chosen three term arithmetic progression $\{0,x,2x\}$.  The set 
$T$ where this convolution is positive will be significantly larger
than $S$, yet will have very few three term arithmetic progressions. 
We will thus be able to deduce, using a quantitative
version of a theorem of Varnavides \cite{varnavides}, that $r_3(N)/N$ 
is much smaller than $r_3(M)/M$ for some $M = (\log N)^{1/16-o(1)}$. 
It is easy to see that this implies that $r_3(N) = o(N)$. 
Alas, the upper bound that our method will produce
for $r_3(N)$ is quite poor, and is of the quality $r_3(N) \ll N/\log_*(N)$,
which nontheless is the sort of bounds produced by the ``triangle-deletion''
proof of Roth's theorem \cite{ruzsa}. 

Many of the other proofs of Roth's theorem, in particular 
\cite{bourgain}, \cite{heathbrown}, \cite{szemeredi}, and
\cite{szemeredi2}, make use of
similar convolution ideas\footnote{In the case of Szemeredi's 
argument \cite{szemeredi2}, the convolution is disguised, but after the
dust has settled, one will see that he convolves with a measure on a 
very long arithmetic progression.  In the case of \cite{heathbrown} and 
\cite{szemeredi}, the arguments can be directly expressed in terms 
of convolution with a measure supported on a long arithmetic progression.}
; however, none of these 
methods convolve with such a short progression as ours 
(three terms only), and none use the result of Varnavides to achieve a density 
increment.  Furthermore, it seems that our method can be generalized
to any context where: (1) the number of three term progressions in a set 
depends only on a small number of Fourier coefficients; and, (2) 
one has a quantitative version of Varnavides's theorem.  This might prove
especially useful in certain contexts, because the particular
sets on which our method achieves a density increment (via Varnavides) 
are unrelated to the particular additive characters where the 
Fourier transform of $S$ is ``large''\footnote{That is, 
the progression to which we pass with each iteration is unrelated to
the additive characters where $\widehat{1_S}$ is ``large''.}.

\section{Notation}

We shall require a modicum of notation: 
given a function $f : \F_p \to [0,1]$, we write
\[ \Lambda(f) := \E_{x,d \in \F_p} f(x)f(x+d)f(x+2d) \]
(where $\E$ represents an averaged sum; thus the $\E$ above represents $p^{-2}\sum$).
Thus $\Lambda$ gives an average of $f$ over three term arithmetic 
progressions; when $f$ is the indicator function of a set $A$, this is 
just the number of progressions in $A$ divided by $p^2$. We shall make 
use of the Fourier transform $\widehat{f} : \F_p \to \C$ of a function 
$f$, given by
\[ \widehat{f}(r) := \E_{x \in \F_p} f(x) e^{2\pi i r x/p}, \]
as well as the easily-verified Parseval's identity 
\[ \sum_{r \in \F_p} |\widehat{f}(r)|^2 = \E_x |f(x)|^2. \]
It is also easy to check that
\begin{align}
\Lambda(f) = \sum_{r \in \F_p} \widehat{f}(r)^2 \widehat{f}(-2r). 
\label{lambdaFormula}
\end{align}
Given a set $T \subseteq \F_p$, we shall furthermore use the notation
$$
\Lambda(T)\ :=\ \Lambda(1_T).
$$
Finally, the notation $\lVert t \rVert_\T$ will be used to denote the 
distance from $t$ to the nearest integer.

\section{Proof of Theorem \ref{main_theorem}}
Let 
\[ \kappa := \limsup_{N \to \infty} r_3(N)/N. \]
We shall show that $\kappa = 0$, which will prove the theorem.

Let $N \geq 2$ be an integer, and then let $p$ be a prime number satisfying
$$
2N < p < 4N.
$$
The fact that such a $p$ exists is of course the content of Bertrand's postulate.

Let $S \subset [N]$ be a set free of three term progressions 
with $|S| = r_3(N)$. 
Thinking of $S$ as a subset of $\F_p$ in the obvious way, we shall 
write $f = 1_S : \F_p \to \{0,1\}$ for the indicator function of $S$. 
Let 
$$ 
R := \{ r \in \F_p : | \widehat{f}(r) | \geq (2\log\log p/\log p)^{1/2} \}.
$$ 
By Parseval's identity, this set of large Fourier coefficients cannot be 
too big; certainly,
$$
|R|\ \leq\ \log p / 2\log\log p. 
$$
We may therefore dilate these points of $R$ to be contained in a short part of 
$\F_p$. Indeed, by Dirichlet's box principle there is an integer dilate $x$
satisfying 
$$
0 < x < p^{1-1/(|R|+1)} \leq p/\log p,
$$
such that for all $r \in R$ we have
\begin{align}
\lVert x r/p \rVert_{\T} \leq p^{-1/(|R|+1)} \leq 1/\log p. \label{dilateBound}
\end{align}
Taking such an $x$, define 
$$
B := \{ 0,x,2x\},
$$
and define $h$ to be the normalised indicator function for $B$, given by
$$
h(n) := p 1_B(n)/3.
$$
Then convolve $f$ with $h$ to produce the new function
$$
g(n) := (f * h)(n) = (f(n) + f(n-x) + f(n-2x))/3. 
$$
Since 
$$
\widehat{f}(r) - \widehat{g}(r)\ =\ \widehat{f}(r)(1-\widehat{h}(r)),
$$
it is easy to check using \eqref{dilateBound} that for all $r \in \F_p$
$$
|\widehat{f}(r) - \widehat{g}(r)|\ \ll\ (\log\log p/\log p)^{1/2}.
$$
From this, along with the Cauchy-Schwarz inequality, Parseval's identity,
and equation \eqref{lambdaFormula}, one can quickly deduce that
$$
  | \Lambda(f) - \Lambda(g) | \ll (\log\log p/\log p)^{1/2}, 
$$
and therefore since $\Lambda(f) \ll 1/p$ (because $S$ is free of three
term arithmetic progressions), we deduce
\begin{align}
\Lambda(g) \ll (\log\log p/\log p)^{1/2}. \label{LambdaFG}
\end{align}

Define 
$$
T := \{ n \in \F_p : g(n) > 0 \},
$$
and note that from (\ref{LambdaFG}), along with the obvious fact 
that $\Lambda(T) \ll \Lambda(g)$, we have
\begin{align}
\Lambda(T) \ll (\log\log p/\log p)^{1/2}. \label{LambdaT}
\end{align}
Furthermore, since $S$ is free of three term progressions 
even in $\F_p$, we must have that $g(n) \leq 2/3$ for 
all $n \in \F_p$. Thus $1_T(n) \geq 3 g(n)/2$ for all $n$, 
immediately implying that 
$|T| \geq 3|S|/2$. The set $T$ would thus serve our 
purposes if it  was not for the fact that it is not 
necessarily contained in $[N]$. However,
since $x \leq p/\log p$, we certainly have the inclusion 
$T \subset [N + 2p/\log p]$.  So, if we let 
$T'$ be those elements of $T$ lying in $[N]$, then 
$$
|T'| = |T| - O(N/\log N)\ {\rm and\ } \Lambda(T') \leq \Lambda(T). 
$$
Hence, for $N$ large enough, $$|T'| \geq 4|S|/3$$ (unless of course 
$r_3(N) = O(N/\log N)$, but then we would be happy anyway).

We have now created a set $T'$, significantly larger than  
$S$, but with only a few more three term progressions. 
The following lemma, a quantitative version of 
Varnavides's theorem, will help us make use of this information. 
The notation $T_3(X)$ denotes the number of three 
term progressions $a, a+d, a+2d$ with $d \geq 1$ in a set $X$ of 
integers.
\begin{lemma}\label{quantitative_varnavides}
For any $1 \leq M \leq N$, and for any set $A \subseteq [N]$, we have 
\begin{align*} T_3(A) \geq \left( \frac{|A|/N-(r_3(M)+1)/M}{M^4} \right)N^2. \end{align*} 
\end{lemma}
Before we prove this, let us see how we can use it to 
finish the proof of our main theorem. Set $M := \lfloor (\log p/\log\log p)^{1/16} \rfloor$ 
and apply the lemma to our set $T'$ to obtain the estimate
\[ 
\Lambda(T') \gg \frac{4|S|/3N - (r_3(M)+1)/M}{M^4}. 
\]
Comparing this to \eqref{LambdaT} (recalling that 
$\Lambda(T') \leq \Lambda(T)$), we conclude that
$$
r_3(N)/N = |S|/N \leq 3r_3(M)/4M + O( (\log\log N/\log N)^{1/4}).
$$
Thus $r_3(N)/N$ is asymptotically decreasing to $0$,
whence $\kappa = 0$.

\begin{proof}[Proof of Lemma \ref{quantitative_varnavides}]
The result will follow from an averaging procedure essentially
contained in \cite{croot:struct}. We include the proof here since our 
formulation is slightly different: we are working over $[N]$ rather than
$\F_p$, and so we have to take into account the inhomogeneity of $[N]$.

Let $k$ be a positive integer. Let $\mathcal{B}$ denote the collection 
of length $M$ arithmetic progressions contained in $[N]$ with common 
difference at most $k$, and let $\mathcal{B}_d$ denote the 
subcollection consisting of such arithmetic progressions with common 
difference $d$. Throughout this proof we restrict ourselves to progressions 
with positive common difference.

We first claim that any 3AP (three term progression) in $[N]$ can occur in 
at most $M^2/4$ progressions in $\mathcal{B}$. To see this, note that if a 
3AP has common difference $d$, then it can occur in at most $M-2$ 
progressions of length $M$ with common difference $d$. Similarly, the 3AP 
can occur in at most $M-2d/n$ $M$-APs with difference $n$ provided $n$ divides $d$ and
$n \geq 2d/(M-1)$, and in no other $M$-APs. Thus the 3AP can occur in no more than
\[ \sum_{1 \leq m \leq (M-1)/2} (M-2m) \leq M^2/4 \]
members of $\mathcal{B}$, as claimed. It follows immediately that 
\begin{align}\label{T3estimate}
 T_3(A) \geq \frac{4}{M^2} \sum_{B \in \mathcal{B}} T_3(A \cap B).
\end{align}
Now if $B$ is an arithmetic progression of length $M$ and $|A \cap B| > r_3(M)$, 
then by definition we have $T_3(A \cap B) \geq 1$. In view of \eqref{T3estimate} 
our aim shall therefore be to estimate the number of such sets $B$; we shall do this 
by looking at progressions of fixed common differences. Indeed, for a 
fixed common difference $d$, every element in the interval 
$I_d := [(M-1)d + 1, N-(M-1)d]$ is contained in precisely $M$ progressions 
in $\mathcal{B}_d$, and so
\[ 
\sum_{B \in \mathcal{B}} |A \cap B| = \sum_{d \leq k} \sum_{a \in A} \sum_{B \in \mathcal{B}_d} 1_B(a) \geq M \sum_{d \leq k} |A \cap I_d|.
\]
Since $|A \cap I_d| \geq |A| - 2(M-1)d$, this quantity is at least $M k(|A|-2 M k)$.
Now let $\mathcal{C} \subset \mathcal{B}$ be the
set of progressions $B$ for which $|A \cap B| > r_3(M)$. We then have 
\[ 
\sum_{B \in \mathcal{B}} |A \cap B| \leq M|\mathcal{C}| + r_3(M)|\mathcal{B}\setminus \mathcal{C}|,
\]
from which it follows that
\[ 
|\mathcal{C}| \geq k(|A| - 2 M k) - |\mathcal{B}|r_3(M)/M. 
\] 
Since $|\mathcal{B}_d| = N-(M-1)d$ for each $d$, the total number of 
progressions $|\mathcal{B}|$ is at most $N k$. Choosing 
$k = \lfloor N/2M^2 \rfloor$ 
we conclude that there must be at least
\[ 
|\mathcal{C}| \geq \left(\frac{|A|/N - r_3(M)/M - 1/M}{4 M^2}\right)N^2 
\]
sets $B$ for which $|A \cap B| > r_3(M)$. The result thus follows 
from \eqref{T3estimate}.
\end{proof}

\section{Acknowledgments}

We would like to thank Ben Green for pointing out that we can take 
$B$ to be a three term progression for our argument -- in a previous
draft we took $B$ to be a $20$-term arithmetic progression, due to 
a small inefficiency in one part of our proof.

\end{document}